\documentclass{article}
\usepackage{amsfonts}
\usepackage{amsmath}
\usepackage{amssymb}
\usepackage{amsthm}
\usepackage[active]{srcltx}
\usepackage[all]{xy}
\usepackage{enumerate}
\newtheorem{theo}{Theorem}[section]
\newtheorem{lema}[theo]{Lemma}
\newtheorem{cor}[theo]{Corollary}
\newtheorem{prop}[theo]{Proposition}
\newtheorem{rem}[theo]{Remark} 
\newtheorem{defi}[theo]{Definition}
\newtheorem{eje}[theo]{Example}
\newcommand{\N}{\mathbb{N}}
\newcommand{\Z}{\mathbb{Z}}
\newcommand{\R}{\mathbb{R}}
\newcommand{\C}{\mathbb{C}}
\newcommand{\Q}{\mathbb{Q}}

\newcommand{\Sp}{\mathbb{S}}

\newcommand{\D}{\mathbb{D}}

\newcommand{\Ho}[3]{\mathrm{H}_{#1}(#2,#3)}
\title{New Open-book Decompositions in Singularity Theory}

\author{
Hayd\'ee Aguilar-Cabrera\footnote{Research partially supported by CONACyT grants U55084 and J49048-F (Mexico), by ECOS-ANUIES grant M06-M02 (France-Mexico) and the Laboratorio Internacional Solomon Lefschetz (France-Mexico).}\\[12pt]
Instituto de Matem\'aticas, Unidad Cuernavaca\\
Universidad Nacional Aut\'onoma de M\'exico\\[12pt]
Institut de Math\'ematiques de Luminy\\
Universit\'e de la M\'editerran\'ee
}

\begin{document}
\maketitle

\begin{abstract}
In this article, we study the topology of real analytic germs $F \colon (\C^3,0) \to (\C,0)$ given by $F(x,y,z)=\overline{xy}(x^p+y^q)+z^r$ with $p,q,r \in \N$, $p,q,r \geq 2$ and $(p,q)=1$. Such a germ gives rise to a Milnor fibration $\frac{F}{\mid F \mid} \colon \Sp^5\setminus L_F \to \Sp^1$. We describe the link $L_F$ as a Seifert manifold and we show that in many cases the open-book decomposition of $\Sp^5$ given by the Milnor fibration of $F$ cannot come from the Milnor fibration of a complex singularity in $\C^3$.
\end{abstract}

\section*{Introduction}
The study of the geometry of isolated complex singularities is an area with great advances in the last decades. A main result and starting point for this topic is Milnor's Fibration Theorem (see \cite{milnor:singular}): Let $f\colon (\C^n,0) \to (\C,0)$ be a complex analytic germ, let
\begin{equation*}
L_f \colon = f^{-1}(0) \cap \Sp^{2n-1}_{\varepsilon} \ ,
\end{equation*}
where $\Sp^{2n-1}_{\varepsilon}$ is a $(2n-1)$-sphere centred at the origin, of radius $\varepsilon$ sufficiently small. Then the map
\begin{equation*}
\phi_f = \frac{f}{|f|} \colon \Sp^{2n-1}_{\varepsilon} \setminus L_f \to \Sp^1
\end{equation*}
is the projection of a $C^{\infty}$ locally trivial fibration.

In fact, if the origin is an isolated singularity, $\phi_f$ induces an open-book decomposition of the sphere $\Sp^{2n-1}_{\varepsilon}$ with binding $L_f$ and whose pages are the fibres of $\phi_f$.

In \cite{milnor:singular}, J.~Milnor shows that some real analytic germs also give rise to fibrations: Let $f \colon (U \subset \R^{n+k},0) \to (\R^k,0)$ be real analytic and a submersion on a punctured neighbourhood $U\setminus\{0\}$ of the origin in $\R^{n+k}$. Let $L_f$ (defined as above) be the link of the singularity and let $N(L_f)$ be a small tubular neighbourhood of the link in the sphere $\Sp^{n+k-1}_{\varepsilon}$. Then there exists a $C^{\infty}$ locally trivial fibration 
\begin{equation*}
\varphi  \colon \Sp^{n+k-1}_{\varepsilon} \setminus N(L_f) \to \Sp^{k-1} \ .
\end{equation*}

The geometry of isolated real singularities has been less studied from this point of view essentially for two reasons: 

The first one is that the hypothesis is very stringent; for example, when $k=2$, the set of critical points of $f$ is, in general, a curve. Therefore it is not easy to find examples of real analytic germs satisfying such condition. The second reason is that even when a function $f$ satisfies this condition, in general it is not true that the projection $\varphi$ can be given by $\frac{f}{\mid f\mid}$ as in the complex case.

Despite these obstacles there is now an important collection of works giving examples of such real germs and studying the geometry of the fibration $\varphi$. See for example \cite{loo:k2}, \cite{MR0365592}, \cite{MR1900787}, \cite{MR2115674}, \cite{Cis09}, \cite{Oka08}, etc.

One of the main challenges is to find examples sufficiently controlled to give rise to a beautiful geometry; \textit{i.e.}, open-book decompositions; but at the same time far from the complex context in order to get geometrical objects which cannot be obtained from complex Milnor fibrations.

For example, consider the family of real analytic germs $f \colon (\C^2 \cong \R^4, 0) \to (\C \cong \R^2, 0)$ defined by $f(x,y)=x^p \bar{y} + \bar{x} y^q$ with $p,q \geq 2$. By \cite{MR1679312}, $f$ has a Milnor fibration with projection $\phi_f = \frac{f}{|f|}$. In \cite{PASJ03} A.~Pichon and J.~Seade prove that the link $L_f$ is isotopic to the link of the holomorphic germ $g(x,y)=xy(x^{p+1}+y^{q+1})$, but the open-book decomposition given by the Milnor fibration of $f$ is not equivalent to the one given by the Milnor fibration of $g$.

In this paper we study the family of real germs $F \colon (\C^3,0) \to (\C,0)$ given by $F(x,y,z)=\overline{xy}(x^p+y^q)+z^r$ with $p,q,r \in \N$, $p,q,r \geq 2$ and $\gcd(p,q)=1$.

We show that $F$ has an isolated singularity at the origin and the projection of the Milnor fibration is given by $\phi_F = \frac{F}{\mid F \mid}$, therefore it gives rise to an open-book decomposition of $\Sp^5$.

We explicitly describe $L_F$ as a Seifert manifold and show that $L_F$ is homeomorphic to the link of a normal complex surface. Finally, our main results exhibit two families of examples among these germs $F$ whose open-book decompositions of $\Sp^5$ cannot appear as Milnor fibrations of holomorphic germs from $\C^3$ to $\C$. We show these open-book decompositions are new due to the following two reasons:

\begin{quotation}
\textbf{Different Binding:} Given the real analytic germ $F \colon (\C^3,0) \to (\C,0)$ defined by
\begin{equation*}
F(x,y,z)= \overline{xy}(x^2+y^3) + z^r \ \text{with} \ r>2 \ ,
\end{equation*}
we prove that $L_F$ cannot be realised as the link of a singularity in $\C^3$; \textit{i.e.} there does not exist a complex analytic germ $G \colon (\C^3,0) \to (\C,0)$ with isolated singularity at the origin such that the link $L_G$ is homeomorphic to the link $L_F$, although $L_F$ is the link of a normal complex surface singularity.
\end{quotation}

\begin{quotation}
\textbf{Different Pages:} Given the real analytic germ $F \colon (\C^3,0) \to (\C,0)$ defined by
\begin{equation*}
F(x,y,z)= \overline{xy}(x^2+y^q) + z^2 \ \text{with} \ q>2 \ ,
\end{equation*}
let $(X,p)$ be a normal Gorenstein complex surface singularity which link is homeomorphic to $L_F$. Then we show that the Milnor fibre of $F$ is not diffeomorphic to any smoothing of $(X,p)$.
\end{quotation}
These results follow from our main Theorems \ref{opb1} and \ref{opb2}.

The paper is organised as follows: In Section \ref{realcrit} we prove, by a direct computation and also using a result of A.~Pichon (see \cite[Th.~5.1]{MR2115674}), that the family of real analytic functions $F(x,y,z)=\overline{xy}(x^p+y^q)+z^r$ with $(x,y,z) \in \C^3$, $p,q,r \in \N$, $\gcd(p,q)=1$ and $r>2$, has an isolated critical point at the origin.

In Section \ref{pwhp}, we show that the link $L_F$ is a Seifert manifold by proving that $F$ is a polar weighted homogeneous polynomial (see \cite{Cis09}, \cite{Oka08}); additionally, by a result of \cite{Cis09} we obtain that the projection of the Milnor fibration of $F$ is given by the map $\phi_F = F / \mid F \mid$; \textit{i.e.} the Milnor fibration of $F$ gives an open book decomposition of $\Sp^5$.

In Section \ref{link} we describe the homeomorphism class of $L_F$ by giving its Seifert invariants; for this we use the functorial property of the rational Euler class of a Seifert manifold (see \cite{MR741334}) and the slice representation given by P. Orlik in \cite{OP72}. Also, by results of W. Neumann in \cite{Neu:calcplumb}, we show that $L_F$ can be seen as the result of plumbing according to some graph $\Gamma$.

In Section \ref{newop} we state and prove Theorems \ref{opb1} and \ref{opb2}. For the case of different binding, we use the fact that the plumbing graph $\Gamma$ is dual to the resolution graph of a normal complex surface singularity. We show that its corresponding canonical class is not integral so the graph cannot be the dual graph of a resolution of a hypersurface singularity. For the case of different pages we use the Join Theorem of \cite{Cis09} to compute the Euler characteristic of the Milnor fibre and we apply the Laufer-Steenbrink formula to see that the Milnor fibre of $F$ cannot appear as a smoothing of a normal Gorenstein complex surface singularity.

\section*{Acknowledgements}
I am most grateful to Anne Pichon and Jos\'e Seade for their supervision and comments on this work. I also want to thank Jos\'e Luis Cisneros-Molina for helpful conversations and remarks.

\section{Isolated critical point}\label{realcrit}
In this section we show that our family of real germs have isolated critical point at the origin; in order to prove this we use the following result.

\begin{prop}\label{singaisl}
Let $h \colon (\R^{n},0) \to (\R^2,0)$ be a real analytic germ and let $r \in \Z^+$. The analytic germ  $H \colon (\R^{n} \times \C,0) \cong (\R^{n+2},0) \to (\R^2,0)$ defined for $(x_1, \ldots, x_n) \in \R^n$ and $z \in \C$ by
\begin{equation*}
H(x_1, \ldots, x_n, z)=h(x_1, \ldots, x_n) +z^r
\end{equation*}
has an isolated singularity at the origin if and only if $h$ has an isolated singularity at the origin.
\end{prop}

\begin{proof}
The jacobian matrix $M$ of $h(x_1, \ldots, x_n)+z^r$ with respect to the coordinates $x_1, \ldots, x_n, z, \bar{z}$ is given by
\begin{equation*}
\begin{pmatrix}
\begin{matrix}
Dh(x_1, \ldots, x_n)
\end{matrix}
\left|
\begin{matrix}
\frac{1}{2}rz^{r-1}  & \frac{1}{2}r\bar{z}^{r-1} \\[7pt]
\frac{1}{2i}rz^{r-1} & -\frac{1}{2i} r\bar{z}^{r-1}\\
\end{matrix}
\right.
\end{pmatrix}
\end{equation*}
where $Dh(x_1, \ldots, x_n)$ is the jacobian matrix of $h$ with respect to the coordinates $x_1, \ldots, x_n$.

Let $\mathcal{P} \colon \R^n \times \C \to \R^n$ be the projection defined by $\mathcal{P}(x_1, \ldots, x_n, z)=(x_1, \ldots, x_n)$.

If $h$ has an isolated singularity at the origin, then $Dh(x_1, \ldots, x_n)$ has rank $2$ in a neighbourhood $W$ of the origin except the origin. Let $(x_1, \ldots, x_n, z) \in \mathcal{P}^{-1}(W)$. If $(x_1, \ldots, x_n) \neq 0$, then $Dh(x_1, \ldots, x_n)$ has rank two. Otherwise $z \neq 0$ and the matrix
\begin{equation*}
\begin{pmatrix}
\frac{1}{2}rz^{r-1}  & \frac{1}{2}r\bar{z}^{r-1} \\[7pt]
\frac{1}{2i}rz^{r-1} & \frac{-1}{2i} r\bar{z}_{n+1}^{r-1}\\
\end{pmatrix}
\end{equation*}
has rank two. Then $h(x_1, \ldots, x_n)+z^r$ has rank two at each point of $\mathcal{P}^{-1}(W)\setminus \{0\}$.

If $h+z^r$ has isolated singularity at the origin, then $M$ has rank $2$ in a neighbourhood $U$ of the origin except at the origin itself; in particular $M$ has rank $2$ at the points $(x_1, \ldots, x_n,0) \in U \setminus \{0\}$; then the matrix $Dh(x_1, \ldots, x_n, z)$ has rank $2$ in the neighbourhood $\mathcal{P}(U)$ of the origin except at the origin.
\end{proof}

\begin{cor}\label{fpa}
Let $F \colon (\C^3,0) \cong (\R^6,0) \to (\C,0) \cong (\R^2,0)$ be the real analytic function defined by
\begin{equation*}
F(x,y,z)=\overline{xy}(x^p+y^q)+z^r \ ,
\end{equation*}
with $p,q,r \in \N$ and $p,q,r \geq 2$. The function $F$ has isolated singularity at the origin if and only if $(p,q) \neq (2,2)$
\end{cor}
\begin{proof}
Set $f(x,y)=xy$ and $g(x,y) = x^p + y^q$. According to Proposition \ref{singaisl} it is equivalent to proving that $h(x,y) =f(x,y) \overline{g(x,y)}$ has an isolated singularity at $0$ if and only if $(p,q)=(2,2)$. Let us give two proofs of this fact.

First we decompose $f$ in its real and imaginary parts and the corresponding jacobian matrix is:
\begin{align*}
 \begin{pmatrix}
                 \frac{p\overline{xy}x^{p-1}+y(\bar{x}^p+\bar{y}^q)}{2} & \frac{\bar{y}(x^p+y^q)+pxy\bar{x}^{p-1}}{2} & \frac{q\overline{xy}y^{q-1}+x(\bar{x}^p+\bar{y}^q)}{2} & \frac{\bar{x}(x^p+y^q)+qxy\bar{y}^{q-1}}{2}\\
                 \frac{p\overline{xy}x^{p-1}-y(\bar{x}^p+\bar{y}^q)}{2i} & \frac{\bar{y}(x^p+y^q)-pxy\bar{x}^{p-1}}{2i} & \frac{q\overline{xy}y^{q-1}-x(\bar{x}^p+\bar{y}^q)}{2i} & \frac{\bar{x}(x^p+y^q)-qxy\bar{y}^{q-1}}{2i}
                 \end{pmatrix} \ . \\
\end{align*}
Then this matrix has rank less than two in a point $(x,y)$ if and only if the following equations, that are the result of calculating the minors of order 2, are satisfied:
\begin{align}
p^2 |xy|^2 |x|^{2(p-1)} & = |y|^2 |x^p+y^q|^2 \label{uno} \\
q^2 |xy|^2 |y|^{2(q-1)} &=  |x|^2 |x^p+y^q|^2 \label{dos} \\
x\bar{y}|x^p+y^q|^2 &= pq |xy|^2\bar{x}^{p-1}y^{q-1} \label{tres} \ .
\end{align}
From these equations we get that the origin $(0,0)$ is always a critical point, and we also get that if $x=0$ then $y=0$ and vice versa.

Therefore, in order to look for another critical points we can suppose $x \ne 0 \ne y$. Simplifying equations \eqref{uno} and \eqref{dos}, we get:
\begin{equation*}
p^2|x|^{2p} = q^2|y|^{2q} = |x^p+y^q|^2 \ .
\end{equation*}
A direct computation shows that this equation together with \eqref{tres} have non trivial solutions if and only if $p = q = 2$.

An alternative proof consists of using (\cite{MR2115674}, Th~5.1), which states that $f\bar{g}$ has an isolated singularity at $0$ if and only if the link $L_f-L_g$ is fibred. Let $\pi : X \rightarrow \C^2$ be the minimal resolution of the germ $fg$. Then (\cite{MR2115674}, 2.1) states that $L_f-L_g$ is fibred if and only if   $m_f - m_g =0$, where $m_f$ and $m_g$ denote the multiplicities of $f \circ \pi$ and $g \circ \pi$ along the (unique in our case) rupture component of $\pi^{-1}(0)$. As $m_f = pq$ and $m_g = p+q$, the corollary is proved.
\end{proof}

\begin{rem}
\begin{itemize}
\item Notice that the arguments we used above to show that $f$ has an isolated critical point are, in this particular case, equivalent to showing that the system given in \cite[page 8]{MR2115674} has non-trivial solutions if and only if $p=q=2$.
\item We notice as well that Cor~\ref{fpa} is consistent with \cite[Example 1.1.c]{PichSea:barfg} where it is shown that the function
\begin{equation*}
\overline{z_1 z_2 \cdots z_n}(z_1^{a_1} + \cdots + z_n^{a_n}) \ , a_i \geq 2 \ ,
\end{equation*}
has $0$ as an isolated critical value if and only if the sum of the $\frac{1}{a_i}$ is not 1.
\end{itemize}
\end{rem}

\section{Polar weighted homogeneous polynomials}\label{pwhp}
In this section we recall the definition of a polar weighted homogeneous polynomial and some properties that we will use later. These polynomials were introduced by Cisneros-Molina in \cite{Cis09} following ideas from Ruas, Seade and Verjovsky in \cite{MR1900787} and studied by Oka in \cite{Oka08} and \cite{oka-2009}.

Let $(p_1, \ldots, p_n)$ and $(u_1, \ldots, u_n)$ in $({\Z^+})^n$ be such that $\gcd(p_1, \ldots, p_n)=1$ and $\gcd(u_1, \ldots, u_n)=1$, we consider the action of $\R^+ \times \Sp^1$ on $\C^n$ defined by:
\begin{equation*}
(t, \lambda) \cdot (z) = (t^{p_1} \lambda^{u_1}z_1, \ldots, t^{p_n}\lambda^{u_n}z_n) \ ,
\end{equation*}
where $t \in \R^+$ and $\lambda \in \Sp^1$. 

\begin{defi}
Let $f \colon \C^n \to \C$ be a function defined as a polynomial function in the variables $z_i, \bar{z_i}$:
\begin{equation*}
f(z_1, \ldots, z_n) = \sum_{\mu, \nu} c_{\mu, \nu} z^{\mu} \bar{z}^{\nu} \ ,
\end{equation*}
where $\mu= (\mu_1, \ldots, \mu_n)$, $\nu=(\nu_1, \ldots, \nu_n)$, with $\mu_i, \nu_i$ not negative integers and $z^{\mu}=z_1^{\mu_1} \cdots z_n^{\mu_n}$ (same for $\bar{z}$).

The function $f$ is called a \textbf{polar weighted homogeneous polynomial} if there exists $(p_1, \ldots, p_n)$ and $(u_1, \ldots, u_n)$ in $({\Z^+})^n$ and $a, c \in \Z^+$ such that the action defined above satisfies the functional equality:
\begin{equation*}
f((t, \lambda) \cdot (z)) = t^a \lambda^c (f(z)) \ ,
\end{equation*}
where $z=(z_1, \ldots, z_n)$.
\end{defi}

It is known that any polar weighted homogeneous polynomial has an isolated critical value at the origin (see \cite{Cis09} and \cite{Oka08}).

\begin{defi}
A \textbf{Seifert manifold} is a closed connected 3-manifold endowed with a fixed-point free action of $\Sp^1$; \textit{i.e.}, for all $x \in M$, there exist $\lambda \in \Sp^1$ such that
\begin{equation*}
\lambda \cdot x \neq x \ .
\end{equation*}
\end{defi}

Thanks to a theorem of Epstein (see \cite{MR0288785}), this is equivalent to requiring that $M$ admits an oriented foliation by circles such that the leaves coincide with the orbits of the $\Sp^1$-action.

In this paper only oriented $3$-manifolds will arise. Thus from now on all Seifert manifolds are oriented.

\begin{prop}
Let $f \colon \C^n \to \C$ be a polar weighted homogeneous polynomial with isolated critical point at the origin. Then there exist a fixed-point free action of $\Sp^1$ on the link $L_f := f^{-1}(0) \cap \Sp^{2n-1}$ induced by the action of $\R^+ \times \Sp^1$ described above. Hence, if $n=3$, $L_f$ is a Seifert manifold.
\end{prop}
\begin{proof}
Given a polar weighted homogeneous polynomial $f$, the action $\cdot$ of $\R^+ \times \Sp^1$ on $\C^n$ induces an action of $\Sp^1$ on $\C^n$ by setting
\begin{equation*}
\lambda \star z := (1, \lambda) \cdot z = (\lambda^{u_1}z_1, \ldots, \lambda^{u_n}z_n) \ ,
\end{equation*}
where $\lambda \in \Sp^1$.

Notice that $f^{-1}(0)$ is invariant under the action $\star$ as well as any sphere $\Sp_{\varepsilon}^{2n-1}$. Therefore, the link $L_f$ is invariant under this action. 

Moreover, this action is a fixed-point free and effective action. Indeed, let $z \in \C^n$ and suppose $z$ is a fixed point of the action $\star$; \textit{i.e.}, for all $\lambda \in \Sp^1$ we have 
\begin{equation*}
\lambda \star z = z \ ,
\end{equation*}
then
\begin{equation*}
f(z) = f(\lambda \star z) = \lambda^c f(z) \ \text{with} \ c \in  \Z^+ \ ,
\end{equation*}
which is a contradiction. When $L_f$ is a $3$-manifold, it is by definition, a Seifert manifold.
\end{proof}

\begin{defi}
An \textbf{open-book decomposition} of a smooth $n$-manifold $M$ consists of a codimension $2$ submanifold $N$ called the \textbf{binding}, embedded in $M$ with trivial normal bundle, together with a fibre bundle decomposition of its complement:
\begin{equation*}
\pi \colon M \setminus N \to \Sp^1 \ ,
\end{equation*}
satisfying that on a tubular neighbourhood of N, diffeomorphic to $N \times \D^2$, the restriction of $\pi$ to $N \times (\D^2 \backslash \{0\})$ is the map $(x,y) \mapsto y / ||y||$. The fibres of $\pi$ are called the \textbf{pages} of the open book. The map $\pi$ is called an \textbf{open-book fibration} of $N$.
\end{defi}

\begin{theo}[{\cite[Prop~3.4]{Cis09}}]
Let $f \colon \C^n \to \C$ be a polar weighted homogeneous polynomial, then the map $\phi$ defined by
\begin{equation*}
\phi= \frac{f}{|f|} \colon (\Sp_{\varepsilon}^{2n-1} \backslash L_f) \to \Sp^1
\end{equation*}
is a locally trivial fibration for any $\varepsilon >0$.
\end{theo}

\begin{cor}
Let $f \colon \C^n \to \C$ be a polar weighted homogeneous polynomial with isolated singular point at the origin. Then the Milnor fibration of $f$ gives an open-book decomposition of $\Sp^{2n-1}$ with the link $L_f$ as binding and the Milnor fibres as pages.
\end{cor}

\begin{prop}\label{fF}
Let $p,q,r \in \Z^+$ such that $\gcd(p,q)=1$. Then the polynomials
\begin{enumerate}[i)]
\item $\overline{xy}(x^p+y^q)$ and \label{fchiq}
\item $\overline{xy}(x^p+y^q) +z^r$
\end{enumerate}
are polar weighted homogeneous polynomials.
\end{prop}
\begin{proof}
Set $\delta=(r,pq-p-q)$; let us consider the following action of $\R^+ \times \Sp^1$ on $\C^2$:
\begin{equation*}
 (t, \lambda) \cdot (x,y) = (t^q \lambda^{\frac{rq}{\delta}} x, t^p \lambda^{\frac{rp}{\delta}} y) \ .
\end{equation*}
let $(x,y) \in \C^2$, for all $(t, \lambda) \in \R^+ \times \Sp^1$ we have that
\begin{equation*}
f((t, \lambda) \cdot (x,y)) = t^{q+p+pq} \lambda^{\frac{r(-q-p+pq)}{\delta}} f(x,y) \ ,
\end{equation*}
Thence the polynomial in \eqref{fchiq} is polar weighted homogeneous.

Now, by Example 2.6 of \cite{Cis09}, the sum of two polar weighted homogeneous polynomials in independent variables is again a polynomial of this type. Then, using that $z^r$ is a polar weighted homogeneous polynomial, we have that $\overline{xy}(x^p+y^q) +z^r$ is also a polar weighted homogeneous polynomial.
\end{proof}

\section{The link as a Seifert manifold}\label{link}
A Seifert manifold is classified by its \textbf{Seifert invariants}: Let $M$ be an oriented Seifert manifold, \textit{i.e.}, a connected closed oriented $3$-manifold endowed with a fixed-point free action of $S^1$.  Let $\pi \colon M \rightarrow B$ be the projection on the orbit space $B$. The space $B$ is a compact connected orientable surface and we denote its genus by $g$.

Given a Seifert manifold $M$, we will call \textbf{Seifert fibration} to the fibration $\pi \colon M \to B$.

The $\Sp^1$-action on $M$ induces an orientation on each orbit. Together with the orientation of $M$, this induces an orientation of $B$ and of each local section of $\pi$, in such a way that the orientation 
of a local section followed by the orientation of the orbits gives the orientation of $M$. 

Given an $\Sp^1$-action, the kernel can only be either the trivial group, a cyclic group $\Z_\sigma$ with $\sigma \geq 2$ or $\Sp^1$ itself. In our case, the kernel cannot be $\Sp^1$ since the action is fixed-point free. If the kernel is the trivial group, then the action is effective and in this case we say that an orbit is \textbf{exceptional} if its isotropy subgroup is non-trivial, which is a finite cyclic subgroup of order $\sigma \geq 2$.

Now, if the kernel is a cyclic group $\Z_\sigma$ with $\sigma \geq 2$, the induced action of the quotient $\Sp^1 \diagup \Z_\sigma$ (which is homomorphic to $\Sp^1$) is effective, then we define an exceptional orbit in the same way as above. Notice there are a finite number of exceptional orbits. 

For each exceptional orbit $O \in B$, there exists a tubular neighbourhood $T$ which is union of orbits and  an orientation preserving diffeomorphism $\phi$ of $T$ with the mapping torus
\begin{equation*}
 \D^2 \times \Z / (\rho(x),t) \sim (x,t+1)
\end{equation*}
of a rotation $\rho$ of order $\alpha$ on the oriented 2-disc $\D^2 = \{ z \in \C \mid |z| \leq 1\}$, sending orbits to orbits preserving their orientations. A disc $\phi^{-1}(\D^2 \times \{t\})$ is called a slice of $O$. As $\phi^{-1}(\D^2)$ is a local section of $\pi$, the previous choices of orientation induce an orientation of $\D^2$. Then the angle of the rotation $\rho$ is well defined. 

Seifert invariants for an exceptional orbit are defined as follows. Suppose that the rotation angle on the 2-disc is equal to $2 \pi \beta^* / \alpha$. We have $\gcd (\alpha,\beta^*) = 1$. Let $\beta$ be any integer such that $\beta \beta^* \equiv 1 \pmod{\alpha}$. The  pair $(\alpha,\beta)$ is a Seifert invariant of the exceptional orbit. See \cite[pages 135-140]{MR915761} for more details.

The choice of $\beta$ in its residue class $(\bmod \ \alpha)$ is related to the choice of a section near the exceptional orbit. The Seifert invariant $(\alpha,\beta)$ is called \textbf{normalised} if a section is chosen in such a way that $0 < \beta < \alpha$.

Let now $O_1,\ldots,O_s \in B$ be the exceptional orbits and let $D_1 \ldots,D_s$ be disjoint open discs in $B$, such that $D_i$ is neighbourhood of the orbit $O_i$. Let us fix a section $R$ of the fibration over $B \setminus (D_1 \cup \cdots \cup D_s)$.  The choice of $R$ fixes the $\beta_i$ for each $O_i$, and also defines an Euler class $e \in \Z$, which is essentially the obstruction to  extending $R$ over the whole surface $B$. For details see \cite{MR915761}.

The integers $e$ and $\beta_1, \ldots, \beta_s$ depend on the choice of the section $R$, but the rational number $e_0 = e-\sum \beta_i / \alpha_i$, called the \textbf{rational Euler class} of the Seifert fibration, does not. 

The (normalised) Seifert invariants of $M$ consists of the data : 
\begin{equation*}
\Bigl(g; \, e_0; \, (\alpha_1, \beta_1), \ldots, (\alpha_s, \beta_s)\Bigr)
\end{equation*}
where $g \geq 0$ and $0 <\beta_i<\alpha_i$. 

The rational Euler number has the property of functoriality as it is shown by Jankins and Neumann in \cite[Th~3.3]{MR741334} and Neumann and Raymond in \cite[Th~1.2]{NeuRay:plumb}:

Let $\pi_1 \colon M_1 \to B_1$ and $\pi_2 \colon M_2 \to B_2$ be Seifert fibrations. Assume there exists a map $p$ such that the diagram
\begin{equation*}
\xymatrix{
M_1 \ar[d]_{\pi_1} \ar[r]^{p} & M_2 \ar[d]^{\pi_2} \\
B_1 \ar[r]^{\tilde{p}} & B_2 \\
}
\end{equation*}
commutes and $\deg(p)=m$ and $\deg(\tilde{p}\mid_{fibre})=n$. Then 
\begin{equation}\label{funct}
e_0(M_1 \xrightarrow{\pi_1} B_1)= \frac{m}{n} e_0(M_2 \xrightarrow{\pi_2} B_2) \ .
\end{equation}

\begin{theo}\label{thprinc}
Let $(p,q)$ be coprime integers and $r \in \N$ with $p,q,r \geq 2$. Let $F \colon \C^3 \to \C$ be the function defined by $F(x,y,z)= \overline{xy}(x^p+y^q)+z^r$. Set $\delta=\gcd(r,pq-p-q)$. The Seifert invariants of the link $L_F$ are
\begin{equation*}
\Bigl(\frac{\delta-1}{2}; \ - \frac{\delta^2}{pqr}; \ (qr/\delta, \beta_1),\,(pr/\delta, \beta_2),\,(r/\delta, \beta_3)\Bigr) \ ,
\end{equation*}
where
\begin{align*}
\frac{pq-p-q}{\delta} \beta_1 & \equiv -1 \pmod{qr/ \delta} \ ,\\
\frac{pq-p-q}{\delta} \beta_2 & \equiv -1 \pmod{pr/ \delta} \ ,\\
\frac{pq-p-q}{\delta} \beta_3 & \equiv 1 \pmod{r/ \delta} \ .
\end{align*}
\end{theo}

\begin{cor}
There exists a normal complex surface singularity $(X,p)$ whose link is homeomorphic to the link $L_F$.
\end{cor}
\begin{proof}
By Theorem \ref{thprinc}, the rational Euler class of $L_F$ (seen as a Seifert manifold) is negative.
By \cite[Th~6]{Neu:calcplumb} we know that a Seifert manifold $M$ has rational Euler class $e_0 < 0$ if and only if it is the link of an isolated normal singularity of a complex surface.

We can conclude the existence of a normal complex surface singularity $(X,p)$ whose link is homeomorphic to $L_F$.
\end{proof}

\begin{proof}[of Theorem \ref{thprinc}]
Let $f \colon (\C^2,0) \to (\C,0)$ be defined by $f(x,y)=\overline{xy}(x^p+y^q)$ with $(p,q)=1$ and $p,q \geq 2$, and let $F=f(x,y)+z^r$ with $r \in \N$ and $r \geq 2$.

Let $\varepsilon$ be such that the sphere $\Sp^5_{\varepsilon}$ is a Milnor ball (see \cite[Th~2.10]{milnor:singular}) for $F$ and let $\varepsilon^{\prime}$ such that for all $(x,y,z) \in F^{-1}(0)$ with $(x,y) \in \D^4_{\varepsilon^{\prime}}$ we have $|f(x,y)|^{1/r} < \varepsilon$.

We consider the polydisc
\begin{equation*}
\D^6 = \{(x,y,z) | (x,y) \in \D^4_{\varepsilon^{\prime}}, |z| \leq \varepsilon \} \ .
\end{equation*}
For technical reasons we replace the sphere $\Sp^5$ by the boundary $\partial \D^6$ of the polydisc $\D^6$.

According to \cite[Th~ 3.5]{durfee:neighalgs}, the link $F^{-1}(0) \cap \Sp^5$ is diffeomorphic to the intersection $F^{-1}(0) \cap \partial \D^6$, then in the sequel we will denote this intersection by $L_F$.

First we see that $L_F$ is a Seifert manifold: Let $\ast$ be the action of $\Sp^1$ on $\C^3$ given by
\begin{equation*}
\lambda \ast (x,y,z) = (\lambda^{\frac{rq}{\delta}} x, \lambda^{\frac{rp}{\delta}} y, \lambda^{\frac{pq-p-q}{\delta}} z) \ .
\end{equation*}
Notice that $\partial \D^6$ and $F^{-1}(0)$ are invariant under the action $\ast$, then so it is the link $L_F$, \textit{i.e.} it is a Seifert manifold. From now on we denote by $\ast$ the restriction of the action to $L_F$.

Let $\bullet$ be the action of $\Sp^1$ on $\C^2$ given by
\begin{equation*}
\lambda \bullet (x,y) = (\lambda^{\frac{rq}{\delta}} x, \lambda^{\frac{rp}{\delta}}y) \ .
\end{equation*}
As $\Sp^3$ is invariant under this action, we denote the restriction of the action to $\Sp^3$ with the same notation. Notice that $L_f$ is invariant by the action $\bullet$ and it consists of three orbits of this action.

Let $\mathcal{P} \colon L_F \to \Sp^3$ be the projection given by
\begin{equation*}
\mathcal{P}(x,y,z)=(x,y) \ .
\end{equation*}
One can see the projection $\mathcal{P}$ is a  cyclic branched  $r$-covering with cover space $L_F$, base space $\Sp^3$ and ramification locus $L_f=f^{-1}(0) \cap \Sp^3$, and it is equivariant with respect to the actions $\ast$ and $\bullet$; i.e.
\begin{equation}\label{equiv}
\mathcal{P}(\lambda \ast (x,y,z)) =  \lambda \bullet \mathcal{P}(x,y) \ .
\end{equation}
A consequence of the equivariance of $\mathcal{P}$ is that the preimage of an orbit of the action $\bullet$ is the disjoint union of orbits of the action $\ast$.

We orient $L_F$ consistently with the orientation of $\Sp^3$ via the projection $\mathcal{P}$. Now, let $B$ the orbit space under the action $\ast$ and let $\pi_{\ast}$, $\pi_{\bullet}$ the projections of the Seifert fibrations $\pi_{\ast} \colon L_F \to B$ and $\pi_{\bullet} \colon \Sp^3 \to \Sp^2$ respectively.

In order to compute the genus of $B$ and the rational Euler class of the Seifert fibration $\pi_{\ast} \colon L_F \to B$ we define the induced map $\mathcal{R} \colon B \to \Sp^2$ in the following way: Let $\mathcal{O} \in L_F$ be an orbit of the action $\ast$, let $b= \pi_{\ast}(\mathcal{O})$ and let $s=\pi_{\bullet}(\mathcal{P}(\mathcal{O}))$, then
\begin{equation*}
\mathcal{R}(b) \colon = s \ .
\end{equation*}

\begin{lema}\label{diagcomm}
The map $\mathcal{R}$ is a cyclic branched $\delta$-covering with ramification locus $\pi_{\bullet}(L_f)$.
\end{lema}
\begin{proof}
Let $s \in \Sp^2\setminus\pi_{\bullet}(L_f)$ and let $(x,y) \in \pi_{\bullet}^{-1}(s)$. Then $\mathcal{P}^{-1}(x,y)$ consists of the $r$ points
\begin{equation*}
(x_0,y_0, w^{1/r} e^{\frac{2 \pi k i}{r}}) \ , \ k=0, \ldots, r-1 \ ,
\end{equation*}
where $w = f(x_0,y_0)$.
An easy computation shows that these $r$ points are distributed in $\delta$ orbits.

Now, let $s \in L_f$ and let $(x,y) \in \pi_{\bullet}^{-1}(s)$. Then $\mathcal{P}^{-1}(x,y)$ consists of the single point $(x,y,0)$.
\end{proof}

It follows that we have a commutative diagram:
\begin{equation*}
\xymatrix{
L_F \ar[d]_{\pi_{\ast}} \ar[r]^{\mathcal{P}} & \Sp^3 \ar[d]^{\pi_{\bullet}} \\
B \ar[r]^{\mathcal{R}} & \Sp^2 \\
}
\end{equation*}
which allows us to compute the rational Euler class $e_0(L_F \rightarrow B)$ and the genus $g(B)$. To compute the genus $g(B)$, we apply the Riemann-Hurwitz formula to the branched covering $\mathcal{R}$. The ramification locus consists of 3 points. Therefore,
\begin{equation*}
2g(B)-2 = \delta [2g(\Sp^2)-2]+3(\delta-1) \ ,
\end{equation*}
then $g(B)= \frac{\delta-1}{2}$.

Now we compute the rational Euler class $e_0(L_f \to B)$: by Lemma \ref{diagcomm}, the degree of the restriction of $\mathcal{P}$ to a regular fibre is $r / \delta$ and the degree of $\mathcal{R}$ is $\delta$; since $e_0(\Sp^3 \rightarrow \Sp^2)=- 1/ pq$, it follows from \eqref{funct} that
\begin{equation*}
e_0(L_F \rightarrow \partial \D^6) = - \frac{\delta^2}{pqr} \ .
\end{equation*}
Let us now describe the exceptional orbits. Let $(x_0,y_0,z_0) \in L_F$ such that $x_0 \neq 0$, $y_0 \neq 0$ and $z_0 \neq 0$, then the isotropy subgroup of $(x_0,y_0,z_0)$ by the action $\ast$ consists of the $\lambda \in \Sp^1$ such that
\begin{equation*}
\lambda^{rq/\delta} = \lambda^{rp/\delta} = \lambda^{(pq-p-q)/\delta} = 1 \ .
\end{equation*}
As $\gcd(rq/\delta, rp/\delta, (pq-p-q)/\delta)=1$, the orbit of $(x_0,y_0,z_0)$ is not exceptional.

If $y_0=0$, then $z_0=0$ and the isotropy subgroup of the corresponding orbit consists of the $\lambda \in \Sp^1$ such that $\lambda^{qr/\delta}=1$; \textit{i.e.} it is $\Z_{qr/\delta}$.

If $x=0$, then $z_0=0$ and the isotropy subgroup of the corresponding orbit consists of the $\lambda \in \Sp^1$ such that $\lambda^{pr/\delta}=1$; \textit{i.e.} it is $\Z_{pr/\delta}$.

If $x_0 \neq 0$, $y_0 \neq 0$ and $z_0=0$, the corresponding isotropy subgroup is $\Z_{r/\delta}$ since it consists of the $\lambda \in \Sp^1$ such that $\lambda^{r/\delta}=1$.

Then we have three exceptional orbits:
\begin{align*}
\mathcal{O}_1 & =\{(x,y,z) \in L_F | y=0, z=0\} \ , \\ 
\mathcal{O}_2 & =\{(x,y,z) \in L_F | x=0, z=0\} \ , \\ 
\mathcal{O}_3 & =\{(x,y,z) \in L_F | x \neq 0, y \neq 0, z=0 \} \ .
\end{align*}
with $\alpha_1=qr/\delta$, $\alpha_2=pr/\delta$ and $\alpha_3=r/\delta$.

Let us now compute each $\beta_i$, following the method that uses the slice representation given by Orlik in \cite[pages 57 and 58]{OP72}.

In order to obtain $\beta_1$, one has to compute the angle of the rotation performed by the first return of the orbits of the $\Sp^1$-action on a slice $D \in L_F$ of $\mathcal{O}_1$, say at $(\epsilon^{\prime},0,0) $, \textit{i.e.} by the rotation performed on $D$ by the action of  $e^{2i\pi / \alpha_1}$.

Instead of dealing with a slice  in $L_F$, let us consider a small disk $D'$ close to $D$ in the intersection of $F^{-1}(0) \cap \{ x = \epsilon'\}$, \textit{i.e.}
\begin{equation*}
D'=\{  \epsilon'^{p} \bar{y} +  \bar{y} y^q + z^r=0\} \ .
\end{equation*}
Very close from $(\epsilon',0,0) $, $D'$ can be approximated by the disk parametrised by $z$ :
\begin{equation*}
D''  = \{(\epsilon', -\bar{z}^r,z), z \in \C, |z| \ll 1)\} \ ,
\end{equation*}
and the action of $e^{\frac{2i\pi}{\alpha_1}}$ on $D''$ can be approximated by : 
\begin{equation*}
e^{\frac{2 \pi i \delta}{qr}} \cdot (1, - \bar{z}^r, z) = \bigl(1, - \overline{ \bigl( (e^{\frac{2 \pi i \delta}{qr}})^{\frac{pq-p-q}{\delta}} z \bigr)}^r, (e^{\frac{2 \pi i \delta}{qr}})^{\frac{pq-p-q}{\delta}} z \bigr) \ ,
\end{equation*}
where the disk $D''$ is invariant under the action. Therefore, $\displaystyle\beta_1^*$ equals $\displaystyle\frac{pq-p-q}{\delta} \pmod{\frac{qr}{\delta}}$ up to sign. 

In order to get the right sign, the disk $D''$ has to be oriented as a complex slice of $\mathcal{P}(\mathcal{O}_1)$ via $\mathcal{P}$.  But
\begin{equation*}
\mathcal{P}(D'') = \{ (\epsilon', -\bar{z}^r), z \in \C\} \ ,
\end{equation*}
which is a slice of the orbit $ \mathcal{P}(\mathcal{O}_1)$ endowed with the opposite orientation as the  complex one. Therefore, we consider $D''$ oriented by $\bar{z}$ (and not $z$). 

We then obtain $\displaystyle\beta_1^* = - \frac{pq-p-q}{\delta} \pmod{\frac{qr}{\delta}}$, and then $\beta_1$ is defined by
\begin{equation*}
- \frac{pq-p-q}{\delta}  \beta_1 \equiv 1 \pmod{\frac{qr}{\delta}} \ .
\end{equation*}
A similar computation leads to
\begin{equation*}
- \frac{pq-p-q}{\delta}  \beta_2 \equiv 1 \pmod{\frac{pr}{\delta}} \ .
\end{equation*}
And for the third orbit $\mathcal{O}_3$, we consider the $z$-plane as the slice, \textit{i.e.} we will parametrise our slice with the coordinate $z$, then we consider the action of $e^{\frac{2 \pi i \delta}{r}}$ given by
\begin{equation*}
e^{\frac{2 \pi i \delta}{r}} \cdot (x,y,z) = (x,y, e^{\frac{2 \pi i \delta}{r} \cdot {\frac{pq-p-q}{r}}} z) \ ;
\end{equation*}
this action is the standard action of type $[r/ \delta, (pq-p-q)/ \delta]$, then $\beta_3$ is defined by the congruence
\begin{equation*}
\frac{pq-p-q}{\delta} \beta_3 \equiv 1 \pmod{r/ \delta} \ .
\end{equation*}
\end{proof}
Now, as in \cite{NeuRay:plumb}, if $b_1, \ldots, b_k \in \Z$, we use the following notation
\begin{equation*}
[b_1, \ldots, b_k] = b_1 - \cfrac{1}{b_2 - \cfrac{1}{\ddots \cfrac{}{-\cfrac{1}{b_k}}}}
\end{equation*}

Let $\Gamma_F$ be the plumbing graph such that $L_F \cong \partial P(\Gamma_F)$, where $P(\Gamma_F)$ is the four-manifold obtained by plumbing $2$-discs bundles according to $\Gamma_F$.

Then
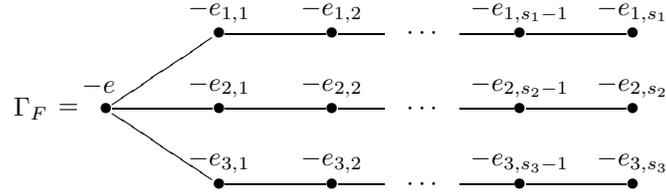
\begin{equation*}
\xy
(10,0)*{\Gamma_F};
(15,0)*{=};
(20,0)*{\bullet}="E";
(19,3)*{-e};
(35,10)*{\bullet}="E11";
(35,13)*{-e_{1,1}};
"E";"E11" **\dir{-};
(35,0)*{\bullet}="E21";
(35,3)*{-e_{2,1}};
"E";"E21" **\dir{-};
(35,-10)*{\bullet}="E31";
(35,-7)*{-e_{3,1}};
"E";"E31" **\dir{-};
(50,10)*{\bullet}="E12";
(50,13)*{-e_{1,2}};
"E11";"E12" **\dir{-};
(50,0)*{\bullet}="E22";
(50,3)*{-e_{2,2}};
"E21";"E22" **\dir{-};
(50,-10)*{\bullet}="E32";
(50,-7)*{-e_{3,2}};
"E31";"E32" **\dir{-};
(57,10)*{}="E1C";
"E12";"E1C" **\dir{-};
(57,0)*{}="E2C";
"E22";"E2C" **\dir{-};
(57,-10)*{}="E3C";
"E32";"E3C" **\dir{-};
(62,10)*{\cdots}="E1P";
(62,0)*{\cdots}="E2P";
(62,-10)*{\cdots}="E3P";
(67,10)*{}="E1D";
(67,0)*{}="E2D";
(67,-10)*{}="E3D";
(75,10)*{\bullet}="E1SM1";
(75,13)*{-e_{1,s_{1}-1}};
(75,0)*{\bullet}="E2SM1";
(75,3)*{-e_{2,s_{2}-1}};
(75,-10)*{\bullet}="E3SM1";
(75,-7)*{-e_{3,s_{3}-1}};
"E1D";"E1SM1" **\dir{-};
"E2D";"E2SM1" **\dir{-};
"E3D";"E3SM1" **\dir{-};
(90,10)*{\bullet}="E1S";
(90,13)*{-e_{1,s_{1}}};
(90,0)*{\bullet}="E2S";
(90,3)*{-e_{2,s_{2}}};
(90,-10)*{\bullet}="E3S";
(90,-7)*{-e_{3,s_{3}}};
"E1SM1";"E1S" **\dir{-};
"E2SM1";"E2S" **\dir{-};
"E3SM1";"E3S" **\dir{-};
\endxy
\end{equation*}
where
\begin{equation*}
-e= e_0 + \sum_{i=1}^{3} \frac{\beta_i}{\alpha_i}-3 \ , \quad \frac{\alpha_i}{\alpha_i - \beta_i}=[e_{i,1}, \ldots, e_{i,s_i}], \ \text{for} \ i=1,2,3 \ .
\end{equation*}

\begin{eje}
Let $F$ be the real polynomial defined by
\begin{equation*}
F(x,y,z)= \overline{xy}(x^2+y^3) + z^r \ \text{with} \ r>2 \ .
\end{equation*}
By Theorem~\ref{thprinc}, the Seifert invariants of $L_F$ are given by
\begin{equation*}
\Bigl(0; \ -\frac{1}{6r}; \ (3r, 3r-1),\, (2r, 2r-1),\, (r, 1)\Bigr) \ .
\end{equation*}
To compute the weights in the plumbing graph, we have:
\begin{align*}
\frac{3r}{3r-(3r-1)} &=[3r] \ , \\
\frac{2r}{2r-(2r-1)} &=[2r] \ , \\
\frac{r}{r-1} &= [\underbrace{2,2,\ldots,2}_{(r-1)}] \ ;
\end{align*}
i.e. the plumbing graph $\Gamma_F$ is given by
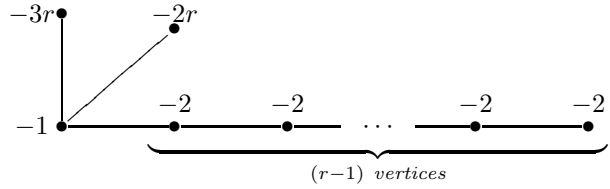
\begin{equation*}
\xy
(20,0)*{\bullet}="E";
(16,0)*{-1};
(20,15)*{\bullet}="T";
(16,15)*{-3r};
"E";"T" **\dir{-};
(35,13)*{\bullet}="E11";
(35,15)*{-2r};
"E";"E11" **\dir{-};
(35,0)*{\bullet}="E21";
(35,3)*{-2};
"E";"E21" **\dir{-};
(50,0)*{\bullet}="E22";
(50,3)*{-2};
"E21";"E22" **\dir{-};
(57,0)*{}="E2C";
"E22";"E2C" **\dir{-};
(62,0)*{\cdots}="E2P";
(67,0)*{}="E2D";
(75,0)*{\bullet}="E2SM1";
(75,3)*{-2};
"E2D";"E2SM1" **\dir{-};
(90,0)*{\bullet}="E2S";
(90,3)*{-2};
"E2SM1";"E2S" **\dir{-};
(62,-5)*{\underbrace{\quad \quad \quad \quad \quad \quad \quad \quad \quad \quad \quad \quad \quad \quad \quad \quad \quad}_{(r-1) \ \text{vertices}}}
\endxy
\end{equation*}
\end{eje}

\begin{eje}
Let $F$ be the polynomial defined by
\begin{equation*}
F(x,y,z)= \overline{xy}(x^2+y^q) + z^2 \ \text{with} \ q>2 \ .
\end{equation*}
We will study two cases.

\noindent \textbf{Case 1.} Let $q=4a+1$ with $a \in \Z^{+}$, then by Theorem~\ref{thprinc}, the Seifert invariants of $L_F$ are given by
\begin{equation*}
\biggl(0; \ -\frac{1}{4(4a+1)}; \ \bigl(2(4a+1), 2a+1 \bigr),\, (4, 1),\, (2, 1)\biggr) \ .
\end{equation*}
Now, the development in continuous fractions for $\alpha_1/(\alpha_1-\beta_1)$ is the following:
\begin{align*}
\frac{8a+2}{6a-1} &= [2,2,4] & \text{when} \ a=1 \ , \\[5pt]
\frac{8a+2}{6a-1} &= [2,2,3,\underbrace{2, \ldots,2}_{(a-2)},3] & \text{when} \ a \geq 2 \ ;
\end{align*}
When $a=1$, the corresponding plumbing graph has the following form:
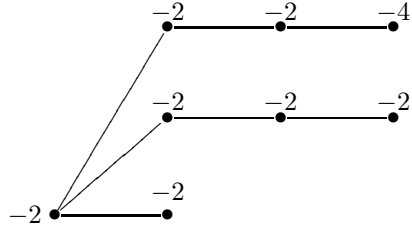
\begin{equation*}
\xy
(20,0)*{\bullet}="E";
(16,0)*{-2};
(35,25)*{\bullet}="E31";
(35,27)*{-2};
"E";"E31" **\dir{-};
(50,25)*{\bullet}="E32";
(50,27)*{-2};
"E31";"E32" **\dir{-};
(65,25)*{\bullet}="E33";
(65,27)*{-4};
"E32";"E33" **\dir{-};
(35,13)*{\bullet}="E11";
(50,13)*{\bullet}="E12";
(35,15)*{-2};
"E";"E11" **\dir{-};
(35,0)*{\bullet}="E21";
(35,3)*{-2};
"E";"E21" **\dir{-};
(50,15)*{-2};
"E11";"E12" **\dir{-};
(65,13)*{\bullet}="E13";
"E12";"E13" **\dir{-};
(65,15)*{-2};
\endxy
\end{equation*}
and if $a \geq 2$, the corresponding plumbing graph has the following form:
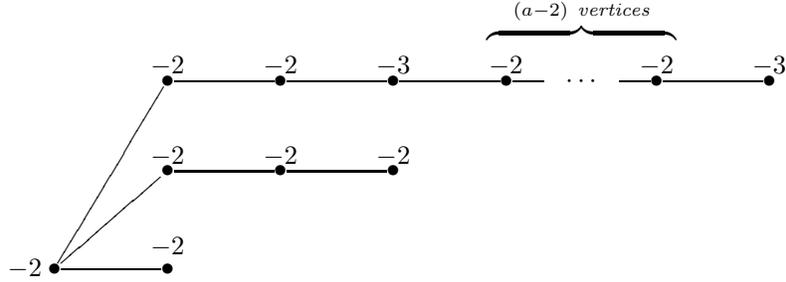
\begin{equation*}
\xy
(20,0)*{\bullet}="E";
(16,0)*{-2};
(35,25)*{\bullet}="E31";
(35,27)*{-2};
"E";"E31" **\dir{-};
(50,25)*{\bullet}="E32";
(50,27)*{-2};
"E31";"E32" **\dir{-};
(65,25)*{\bullet}="E33";
(65,27)*{-3};
"E32";"E33" **\dir{-};
(80,25)*{\bullet}="E34";
(80,27)*{-2};
"E33";"E34" **\dir{-};
(85,25)*{}="E35";
"E34";"E35" **\dir{-};
(90,25)*{\cdots};
(95,25)*{}="E36";
(100,25)*{\bullet}="E37";
(100,27)*{-2};
"E36";"E37" **\dir{-};
(115,25)*{\bullet}="E38";
(115,27)*{-3};
"E37";"E38" **\dir{-};
(35,13)*{\bullet}="E11";
(50,13)*{\bullet}="E12";
(35,15)*{-2};
"E";"E11" **\dir{-};
(35,0)*{\bullet}="E21";
(35,3)*{-2};
"E";"E21" **\dir{-};
(50,15)*{-2};
"E11";"E12" **\dir{-};
(65,13)*{\bullet}="E13";
"E12";"E13" **\dir{-};
(65,15)*{-2};
(90,33)*{\overbrace{\quad \quad \quad \quad \quad \quad \quad}^{(a-2) \ \text{vertices}}}
\endxy
\end{equation*}

\noindent \textbf{Case 2.} Let $q=4a-1$ with $a \in \Z^{+}$; then by Theorem~\ref{thprinc}, the Seifert invariants of $L_F$ are given by
\begin{equation*}
\biggl(0; \ -\frac{1}{4(4a-1)}; \ \bigl(2(4a-1), 6a-1 \bigr),\, (4, 3), \, (2, 1)\biggr) \ .
\end{equation*}
The development in continuous fractions for $\alpha_1/(\alpha_1-\beta_1)$ is the following:
\begin{align*}
\frac{8a-2}{2a-1} &= [6] & \text{when} \ a=1 \ , \\[5pt]
\frac{8a-2}{2a-1} &= [5,3,\underbrace{2, \ldots,2}_{(a-2)},3] & \text{when} \ a \geq 2 \ .
\end{align*}
When $a=1$, the corresponding plumbing graph has the following form:
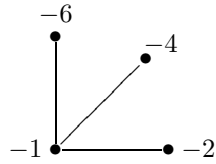
\begin{equation*}
\xy
(0,0)*{\bullet}="E";
(-4,0)*{-1};
(0,15)*{\bullet}="E1";
(0,18)*{-6};
(12,12)*{\bullet}="E2";
(14,14)*{-4};
(15,0)*{\bullet}="E3";
(19,0)*{-2};
"E";"E1" **\dir{-};
"E";"E2" **\dir{-};
"E";"E3" **\dir{-};
\endxy
\end{equation*}
On the other hand, when $a \geq 2$, the corresponding plumbing graph has the following form:
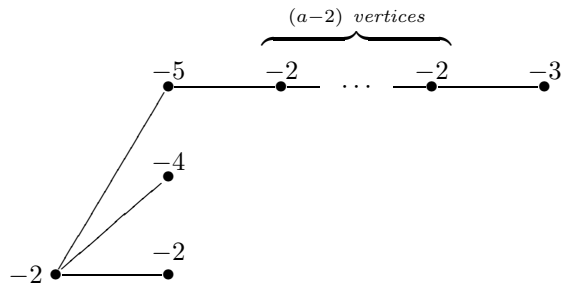
\begin{equation*}
\xy
(20,0)*{\bullet}="E";
(16,0)*{-2};
(35,25)*{\bullet}="E31";
(35,27)*{-5};
"E";"E31" **\dir{-};
(50,25)*{\bullet}="E32";
(50,27)*{-2};
"E31";"E32" **\dir{-};
(55,25)*{}="E33";
"E32";"E33" **\dir{-};
(60,25)*{\cdots};
(65,25)*{}="E35";
(70,25)*{\bullet}="E37";
(70,27)*{-2};
"E35";"E37" **\dir{-};
(85,25)*{\bullet}="E38";
(85,27)*{-3};
"E37";"E38" **\dir{-};
(35,13)*{\bullet}="E11";
(35,15)*{-4};
"E";"E11" **\dir{-};
(35,0)*{\bullet}="E21";
(35,3)*{-2};
"E";"E21" **\dir{-};
(60,33)*{\overbrace{\quad \quad \quad \quad \quad \quad \quad}^{(a-2) \ \text{vertices}}}
\endxy
\end{equation*}
\end{eje}

\section{New open-book decompositions}\label{newop}
In the previous section we gave, for the singularities we envisage in this article, the plumbing graphs corresponding to the Seifert invariants given by the $S^1$-action; we know that this graph has negative definite intersection matrix.

Let us consider now, for a moment, the general setting of an arbitrary plumbing graph $\Gamma$, where each vertex $e_i$ has been assigned a genus $g_i \geq 0$ and a weight $w_i < 0$, so that the corresponding intersection matrix $A_\Gamma$ is negative definite. 

The graph determines a $4$-dimensional compact manifold $\widetilde V_\Gamma$ with boundary $L_\Gamma$ which can be assumed to be an almost-complex manifold and its interior is a complex manifold.

The manifold $\widetilde V_\Gamma$ contains in its interior the exceptional divisor $E$, given by the plumbing description, and $\widetilde V_\Gamma$ has $E$ as a strong deformation retract. In this setting, the intersection matrix corresponds to the intersection product in $\Ho{2}{\widetilde V_\Gamma}{\Q} \cong \Q^n$, where $n$ is the number of vertices in the graph. The generators are the irreducible components $E_i$ of the divisor $E$; each of these is a compact (non-singular) Riemann surface embedded in $\widetilde V_\Gamma$.

The fact that the matrix $A_\Gamma$ is negative definite implies, by Grauert's contractibility criterion (see for instance \cite{MR749574}), that the divisor $E$ can be blown down to a point, and we get a complex surface $V_\Gamma$ with a normal singularity at $0$, the image of the divisor $E$. Since $ V_\Gamma$ is homeomorphic to the cone over $L_\Gamma$, its topology does not depend on the various choices and it is determined by the topology of $L_\Gamma$. However, in general, its complex structure does depend on the various choices. 

\begin{defi}
A resolution $\pi \colon \widetilde{V} \to V$ is \textbf{good} if each irreducible component $E_i$ of $E$ is non-singular and for all $i,j$ with $i \neq j$, $E_i \cap E_j$ is at most one point, and no three of them intersect. 
\end{defi}

Since any good resolution has associated a plumbing graph, in order to simplify the notation, we omit the subindex $\Gamma$ and we just write $A$, $\widetilde V$, $V$ and $L$. 

The manifold $\widetilde V$ has an associated \textit{canonical class} $K$:
\begin{defi}
Given a normal surface singularity $(V,p)$ and a good resolution $\widetilde{V}$ of $V$, the \textbf{canonical class} $K$ of $\widetilde{V}$ is the unique class in $\Ho{2}{\widetilde{V}}{\Q} \cong \Q^n$ that satisfies the \textit{Adjunction Formula}
\begin{equation*}
2g_i-2=E_i^2 + K \cdot E_i
\end{equation*}
for each irreducible component $E_i$ of the exceptional divisor.
\end{defi}

Note that the canonical class $K$ of a good resolution $\widetilde{V}$ only depends on the topology of the resolution. The existence and uniqueness of this class comes from the fact that the matrix $A$ is non-singular.

Since the canonical class $K$ is by definition a homology class in $\Ho{2}{\widetilde V}{\Q}$, it is a rational linear combination of the generators: 
\begin{equation*}
 K = \sum_{i=1}^n k_i \, E_i \ , \text{with} \ k_i \in \Q \ .
\end{equation*}

\begin{defi}(see for instance \cite{DAH78}, Def~1.2)
A normal surface singularity germ $(V,0)$ is \textbf{Gorenstein} if there is a nowhere-zero holomorphic two-form on the regular points of $V$. In other words, its canonical bundle $\mathcal{K}:= \bigwedge^2 \big(T^*(V\setminus\{0\}) \big)$ is holomorphically trivial in a punctured neighbourhood of $0$.
\end{defi}

For instance, if $V$ can be defined by a holomorphic map-germ $f$ in $\C^3$, then the gradient $\nabla f$ is never vanishing away from $0$ and we can contract the holomorphic 3-form $dz_1 \wedge dz_2 \wedge dz_3$ with respect to $\nabla f$ to get a never vanishing holomorphic 2-form on a neighbourhood of $0$ in $ V$. So every isolated hypersurface singularity is Gorenstein. More generally, A. Durfee in \cite{DAH78} introduced the following concept.

\begin{defi} The singularity germ  $(V,0)$ is \textbf{numerically Gorenstein} if the canonical class $K$ of some good resolution $\widetilde V$ is integral.
\end{defi}

\begin{rem}\label{integral}
It is easy to show (see \cite{DAH78} or \cite{MR1359516}) that the condition of $K$ being an integral class  is satisfied if and only if the canonical bundle $\mathcal{K}$ is topologically trivial. Therefore  numerically-Gorenstein is a condition independent of the choice of resolution (whose dual graph is $\Gamma$)  on the germ $(V,0)$ and every Gorenstein germ is numerically Gorenstein. 
\end{rem}

Notice that the canonical class $K$ is determined by the intersection matrix $A$, and therefore it is independent of the complex structure we put on $\widetilde V$. However by definition the class $K$ is associated to the graph that defines the surface $V$, and so does its self-intersection number $K^2$, which is defined in the obvious way. There are many different graphs producing the same singularity germ, and each has a different canonical class, with a different self-intersection number.
 
Similarly, one has another number, an integer, associated to $\widetilde V$: its \textit{Euler-Poincar\'e characteristic} $\chi(\widetilde V)$. Of course this number also depends on the graph $\Gamma$ and not only on the topology of the link $L$ of $ V$. Yet one has the following result, which extends \cite[Remark 7.6.i, p. 125]{Sea}, where this is discussed for Gorenstein singularities.
  
\begin{prop}\label{independent}
Assume the germ $(V,0)$ is numerically Gorenstein and let $\widetilde V$ be a good resolution. Let $L$ be the link of $(V,0)$. Then the integer
\begin{equation*}
\chi(\widetilde V) + K^2
\end{equation*}
is independent of the choice of resolution of $(V,0)$. Moreover, if the link $L_*$ of another isolated surface singularity germ is orientation preserving homeomorphic to $L$, then one has:
\begin{equation*}
\chi(\widetilde V) + K^2 = \chi(\widetilde V_*) + K^2_* \ .
\end{equation*}
 \end{prop}
This theorem is essentially well-known and it can be proved in several ways. The first statement can be proved by direct computation, showing that each time we blow up a smooth point of the exceptional divisor, the Euler-Poincar\'e characteristic of the resolution increases by 1, while the self-intersection of the canonical class diminishes by one. This computation is straight-forward but not easy.

Alternatively, the same statement can be proved by noticing that for a compact complex surface $M$, the invariant $\chi(M) +K(M)^2$ is 12-times the Todd genus $Td(M)$, which is a birational invariant, by   Hirzebruch-Riemann-Roch's Theorem for compact complex surfaces. One can then compactify the surface $\widetilde V$ by adding a divisor at infinity, whose singularities can be resolved, thus getting a smooth compactification $\widehat V$ of $\widetilde V$. The fact that the singularity is numerically  Gorenstein ensures that the Todd genus $Td(\widehat V)$ splits in two parts: one of these is $\chi(\widetilde V) +K^2$. The result then follows from the fact that $Td(\widehat V)$ remains constant under blowing ups.

We finally remark that for singularities that are Gorenstein and smoothable, the fact that $\chi(\widetilde V) +K^2$ does not depend on the choice of resolution is also a direct consequence of the Laufer-Steenbrink formula that we explain below, and the fact that the geometric genus is independent of the choice of resolution.

The second statement in this theorem is now an immediate consequence of the previous statement, together with Neumann's Theorem \cite[Th.~2]{Neu:calcplumb}, that if two normal surface singularities have orientation preserving homeomorphic links, then their minimal resolutions are homeomorphic.

So, we denote this invariant just by $\chi_L + K_L$, since it depends only on the link $L$.

\begin{defi}
A normal surface singularity germ $(V,0)$ of dimension $n \geq 1$ is \textbf{smoothable} if there exists a complex analytic space $(W,0)$ of dimension $n+1$ and a proper analytic map:
\begin{equation*}
\mathcal{F} \colon W \to \D \subset \C
\end{equation*}
where $\D$ is an open disc with centre at $0$, such that:
\begin{enumerate}[i)]
\item it is not a zero divisor in the local ring of $W$ at $0$, i.e. it is not flat;
\item $\mathcal{F}^{-1}(0)$ is isomorphic to $V$; and
\item $\mathcal{F}^{-1}(t)$ is non-singular for $t \neq 0$.
\end{enumerate}
\end{defi}

Following \cite{DAH78}, \cite{MR713277} and \cite{MR0450287}, we call the manifold $\mathcal{F}^{-1}(t)$ a \textbf{smoothing} of $V$. Notice that if $(V,0)$ is a hypersurface germ, then it is always smoothable and the smoothing $\mathcal{F}^{-1}(t)$ is the Milnor fibre.

Now suppose we have a normal Gorenstein surface singularity germ $(V,0)$ and suppose it is smoothable. Let $\widetilde V$ be a resolution of $(V,0)$ and let $V^\#$ be a smoothing. Then the Laufer-Steenbrink formula (see \cite{MR0450287} and \cite{MR713277}) states:
\begin{equation*}
\chi(V^\#) = \chi(\widetilde V) + K^2 + 12 \rho_g(V,0) \ ,
\end{equation*}
where $K$ is the canonical class of the resolution and $\rho_g(V,0)$ is the geometric genus, which is an integer, independent of the choice of resolution (see \cite{MR0199191}). Then we have the following result.

\begin{theo}(see for instance \cite[\S~4, Cor~1]{MR713273})\label{cong}
Let $(V,0)$ be a normal Gorenstein complex surface singularity with link $L$. If $(V,0)$ is smoothable, then one has
\begin{equation*}
\chi_{L} + K_L \equiv \chi(V^{\prime}) \pmod{12}
\end{equation*}
where $V^{\prime}$ is a smoothing of $V$ and $\chi_L + K_L$ is the invariant of $L$ previously defined.
\end{theo}

\begin{theo}\label{opb1}
Let $F \colon (\C^3,0) \cong (\R^6,0) \to (\C,0) \cong (\R^2,0)$ be the real polynomial function defined by
\begin{equation*}
F(x,y,z)= \overline{xy}(x^2+y^3) + z^r \ \text{with} \ r>2 \ .
\end{equation*}
Let $(V,p)$ a complex analytic germ such that $L_V \cong L_F$; then $(V,p)$ is not numerically Gorenstein.
\end{theo}

\begin{cor}
There not exist a complex analytic germ $G \colon (\C^3,0) \to (\C,0)$ with isolated singularity at the origin such that the link $L_G$ is isomorphic to the link $L_F$.
\end{cor}
\begin{proof}[of Theorem \ref{opb1}]
The linear system we need to solve in order to compute the canonical class $K$ of the resolution $\tilde{V}_F$ is of the form:
\begin{align*}
k-k_{1,1}-k_{2,1}-k_{3,1}       & =1 \ , \\[5pt]
-k+3r k_{1,1}                   & =2-3r \ , \\[4pt]
-k_+2 rk_{2,1}                  & =2-2r \ , \\[5pt]
-k+2 k_{3,1}-k_{3,2}            & =0 \ , \\
-k_{3,1}+2k_{3,2}-k_{3,3}       & =0 \ , \\
\cdots & \\
-k_{3,r-3}+2k_{3,r-2}-k_{3,r-1} & =0 \ , \\
-k_{3,r-2}+2k_{3,r-1}           & =0 \ . \\
\end{align*}
We solve the system and we obtain:
\begin{equation*}
\begin{array}{r@{}lcr@{}lcr@{}lcr@{}l}
k &= 10-6r \ , && k_{1,1} &= \frac{4}{r}-3 \ , && k_{2,1} &= \frac{6}{r}-4 \ , && k_{3,1}   &= -\frac{10}{r} -6r +16 \ , \\[5pt]
  &        &&         &                &&         &                &&    \cdots &                 \\[5pt]
  &        &&         &                &&         &                && k_{3,r-2} &= \frac{20}{r} -12 \ , \\[5pt]
  &        &&         &                &&         &                && k_{3,r-1} &= \frac{10}{r} -6 \ .
\end{array}
\end{equation*}
Then $k_{1,1}$ is an integer only if $r=4$ (since $r \neq 2$), but then $k_{2,1}$ is not an integer. We can conclude that the canonical class $K$ has not integer coefficients.
\end{proof}

\begin{theo}\label{opb2}
Let $F \colon (\C^3,0) \cong (\R^6,0) \to (\C,0) \cong (\R^2,0)$ be the real polynomial function defined by
\begin{equation*}
F(x,y,z)= \overline{xy}(x^2+y^q) + z^2 \ \text{with} \ q>2 \ .
\end{equation*}
Then the Milnor fibre $\mathcal{F}$ of $F$ is not the smoothing of a normal Gorenstein complex surface singularity $(X,p)$.
\end{theo}

\begin{cor}
The open-book decomposition of the sphere $\Sp^5$ given by the Milnor fibration of $F$ is not given by the Milnor fibration of a normal Gorenstein complex surface singularity $(X,p)$.
\end{cor}

\begin{proof}[of Theorem \ref{opb2}]
We will consider two cases:

\noindent \textbf{Case 1.} Let $q=4a+1$, then we solve the linear system in order to get the canonical class $K$ and we obtain:
\begin{align*}
k &= -4a \ , & k_{1,1} &= -3a \ ,                        & k_{2,1} &= -3a \ , & k_{3,1} &= -2a \ .\\
  &          & k_{1,2} &= -2a \ ,                        & k_{2,2} &= -2a \ , &         &     \\
  &          & k_{1,j} &= -a+(j-3) \, , \ 3 \leq j \leq a+2 & k_{2,3} &= -a  \ , &         &
\end{align*}

Since the canonical class $K$ has integer coefficients, there could exists a complex analytic germ $G \colon (\C^3,0) \to (\C,0)$ with isolated singularity with link $L_G$ homeomorphic to the link $L_F$.

Now we proceed to see if the open book fibrations given by $F$ and $G$ are equivalent.

In order to see if the  Milnor fibre $\mathcal{F}$ is diffeomorphic to a smoothing of a normal Gorenstein complex surface singularity we compute the Euler characteristic of  $\mathcal{F}$ and the Euler characteristic of the resolution $\tilde{V}_F$  and we apply Theorem \ref{cong}.

We compute the Euler characteristic of $\mathcal{F}$ using the Join Theorem for polar weighted homogeneous polynomials (given in \cite{Cis09}); this result says that the Milnor fibre of the sum (over independent variables) of two polar weighted homogeneous polynomials is homotopically equivalent to the join of the Milnor fibres of the polar weighted homogeneous polynomials. Also we have
\begin{equation*}
\chi(\tilde{V})= \sum_{i=1}^n \chi(E_i) - \sum_{i<j} E_i \cdot E_j \ , 
\end{equation*}
where $E_i$ is one of the irreducible components of $\tilde{V}$ and
\begin{equation*}
K^2 = K^T A K \ ,
\end{equation*}
where $A$ is the intersection matrix of the plumbing graph $\Gamma$.

In our case, the plumbing graph has $a+7$ vertices and $a+6$ edges, then
\begin{align*}
\chi(\tilde{V}_{F}) &= \chi(E) + \sum_{\substack{i=1,j=1}}^{3,a+2} \chi(E_{i,j}) -\sharp(E \cap E_{i,1})_{i=1,2,3} - \sharp(E_{i,j} \cap E_{i, j^{\prime}})_{i=1,2,3, j \neq j^{\prime}} \\
                    &=(a+7)(2)-(a+6)=a+8 \ .
\end{align*}
Also
\begin{align*}
K^2 &= K^T A K \\
    &=(0,0,0,1,\underbrace{0,\ldots,0}_{(a-2)},1,0,0,0)\\
    &=k_{1,3}+k_{1,a+2}= -a + \bigr(-a+(a+2-3)\bigl)\\
    &=-(a+1) \ .
\end{align*}
i.e. we have
\begin{equation*}
\chi(\tilde{V}_F) + K^2 = (a+8)-(a+1)=7 \ .
\end{equation*}

When $a=1$, the resolution of $f(x,y)=\overline{xy}(x^2+y^5)$ is given by the graph
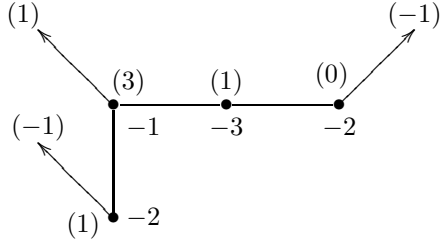
\begin{equation*}
\xy
(0,0)*{\bullet}="E4";
(4,-3)*{-1};
(2,3)*{(3)};
{\ar (0,0)*{}; (-10,10)*{}};
(-12,12)*{(1)};
(0,-15)*{\bullet}="E1";
(4,-15)*{-2};
(-4,-16)*{(1)};
{\ar (0,-15)*{}; (-10,-5)*{}};
(-10,-3)*{(-1)};
(15,0)*{\bullet}="E3";
(15,-3)*{-3};
(15,3)*{(1)};
(30,0)*{\bullet}="E2";
(30,-3)*{-2};
(29,4)*{(0)};
{\ar (30,0)*{}; (40,10)*{}};
(40,12)*{(-1)};
"E4";"E1" **\dir{-};
"E4";"E3" **\dir{-};
"E3";"E2" **\dir{-};
\endxy
\end{equation*}
On the other hand, when $a \geq 2$, the resolution of $f(x,y)=\overline{xy}(x^2+y^{4a+1})$ is given by the graph
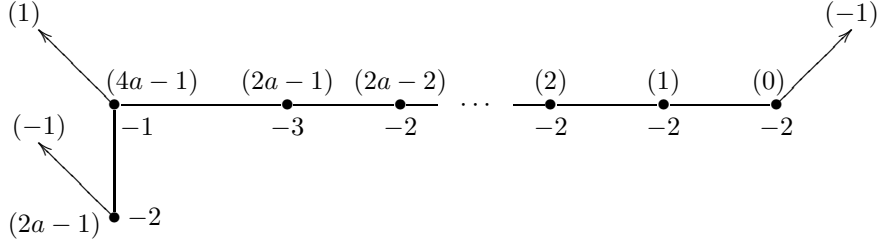
\begin{equation*}
\xy
(-8,0)*{\bullet}="E4";
(-5,-3)*{-1};
(-3,3)*{(4a-1)};
{\ar (-8,0)*{}; (-18,10)*{}};
(-20,12)*{(1)};
(-8,-15)*{\bullet}="E1";
(-4,-15)*{-2};
(-16,-16)*{(2a-1)};
{\ar (-8,-15)*{}; (-18,-5)*{}};
(-18,-3)*{(-1)};
(15,0)*{\bullet}="E3";
(15,-3)*{-3};
(15,3)*{(2a-1)};
(30,0)*{\bullet}="E2";
(30,-3)*{-2};
(30,3)*{(2a-2)};
(35,0)*{}="E5";
"E5";"E2" **\dir{-};
(40,0)*{\cdots};
(45,0)*{}="E6";
(50,0)*{\bullet}="E7";
"E6";"E7" **\dir{-};
(50,-3)*{-2};
(50,3)*{(2)};
(65,0)*{\bullet}="E8";
"E7";"E8" **\dir{-};
(65,-3)*{-2};
(65,3)*{(1)};
(80,0)*{\bullet}="E9";
"E8";"E9" **\dir{-};
(80,-3)*{-2};
(79,3)*{(0)};
{\ar (80,0)*{}; (90,10)*{}};
(90,12)*{(-1)};
"E4";"E1" **\dir{-};
"E4";"E3" **\dir{-};
"E3";"E2" **\dir{-};
\endxy
\end{equation*}

Let $\mathcal{F}_f = \bigvee_{i=1}^k \Sp^1_i$ be the Milnor fibre of $f(x,y)=\overline{xy}(x^p+y^q)$ and let $\mathcal{F}_g$ be the Milnor fibre of the function $g(z)=z^r$ with $z \in \C$; then we have
\begin{equation*}
\chi(\mathcal{F}) = \chi(\mathcal{F}_f \ast \mathcal{F}_g) = \chi \Biggl(\bigvee_{j=1}^{(r-1)k} \Sp^2_j \Biggr) \ .
\end{equation*}
Then, for both cases we have
\begin{equation*}
\chi(\mathcal{F})= (2-1) \chi(V_f)=(2-1)(4a-1)=1-4a \ ,
\end{equation*}
and when
\begin{align*}
\qquad \qquad \qquad a &\equiv 1 \pmod{3} \ , & 1-4a &\equiv 9 \pmod{12} \ , \qquad \qquad \qquad \\
\qquad \qquad \qquad a &\equiv 2 \pmod{3} \ , & 1-4a &\equiv 5 \pmod{12} \ , \qquad \qquad \qquad \\
\qquad \qquad \qquad a &\equiv 0 \pmod{3} \ , & 1-4a &\equiv 1 \pmod{12} \ . \qquad \qquad \qquad
\end{align*}
Thus, the Milnor fibre $\mathcal{F}$ is not diffeomorphic to a smoothing of a normal Gorenstein complex surface singularity.

\noindent \textbf{Case 2.} Let $q=4a-1$ with $a \in \Z^{+}$, then the solutions to the linear system which gives the canonical class $K$ are:
\begin{align*}
k &= -2(2a-1) \ , & k_{1,1} &= -a \ ,                            & k_{2,1} &= -a \ ,  & k_{3,1} &= -2a+1 \ . \\
  &               & k_{1,j} &= -a+(j-1) \, , \ \ 2 \leq j \leq a &         &          &         & \\
\end{align*}

Since the canonical class $K$ is integral, there could be a complex analytic germ $G \colon (\C^3,0) \to (\C,0)$ with isolated singularity with link $K_G$ homeomorphic to the link $L_F$.

Now we proceed to see if the open book fibrations given by $F$ and $G$ are equivalent as before. Firstly, the plumbing graph has $a+3$ vertices and $a+2$ edges, then
\begin{align*}
\chi(\tilde{V}_{F}) &= \chi(E) + \sum_{\substack{i=1,j=1}}^{3,a} \chi(E_{i,j}) -\sharp(E \cap E_{i,1})_{i=1,2,3} - \sharp(E_{i,j} \cap E_{i, j^{\prime}})_{i=1,2,3, j \neq j^{\prime}} \\
                    &=(a+3)(2)-(a+2)=a+4 \ .
\end{align*}
Also
\begin{align*}
K^2 &= K^T A K \\
    &=(-1,3,\underbrace{0,\ldots,0}_{(a-2)},1,2,0)\\
    &=-k+3k_{1,1}+k_{1,a}+2k_{2,1}= 2(2a-1)+3(-a)+(-1)+2(-a)\\
    &=-(a+3) \ .
\end{align*}
i.e. we have
\begin{equation*}
\chi(\tilde{V}_F) + K^2 = (a+4)-(a+3)=1 \ .
\end{equation*}

When $a=1$, the resolution of $f(x,y)=\overline{xy}(x^2+y^3)$ is given by the graph
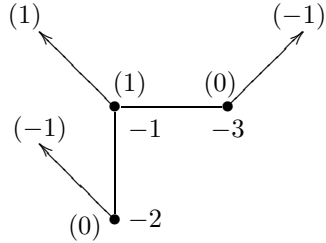
\begin{equation*}
\xy
(0,0)*{\bullet}="E4";
(4,-3)*{-1};
(2,3)*{(1)};
{\ar (0,0)*{}; (-10,10)*{}};
(-12,12)*{(1)};
(0,-15)*{\bullet}="E1";
(4,-15)*{-2};
(-4,-16)*{(0)};
{\ar (0,-15)*{}; (-10,-5)*{}};
(-10,-3)*{(-1)};
(15,0)*{\bullet}="E3";
(15,-3)*{-3};
(14,3)*{(0)};
{\ar (15,0)*{}; (25,10)*{}};
(24.5,12)*{(-1)};
"E4";"E1" **\dir{-};
"E4";"E3" **\dir{-};
\endxy
\end{equation*}
Also, when $a \geq 2$, the resolution of $f(x,y)=\overline{xy}(x^2+y^{4a-1})$ is given by the graph
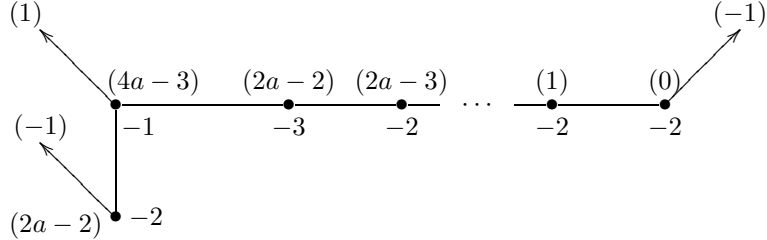
\begin{equation*}
\xy
(-8,0)*{\bullet}="E4";
(-5,-3)*{-1};
(-3,3)*{(4a-3)};
{\ar (-8,0)*{}; (-18,10)*{}};
(-20,12)*{(1)};
(-8,-15)*{\bullet}="E1";
(-4,-15)*{-2};
(-16,-16)*{(2a-2)};
{\ar (-8,-15)*{}; (-18,-5)*{}};
(-18,-3)*{(-1)};
(15,0)*{\bullet}="E3";
(15,-3)*{-3};
(15,3)*{(2a-2)};
(30,0)*{\bullet}="E2";
(30,-3)*{-2};
(30,3)*{(2a-3)};
(35,0)*{}="E5";
"E5";"E2" **\dir{-};
(40,0)*{\cdots};
(45,0)*{}="E6";
(50,0)*{\bullet}="E7";
"E6";"E7" **\dir{-};
(50,-3)*{-2};
(50,3)*{(1)};
(65,0)*{\bullet}="E8";
"E7";"E8" **\dir{-};
(65,-3)*{-2};
(65,3)*{(0)};
{\ar (65,0)*{}; (75,10)*{}};
(75,12)*{(-1)};
"E4";"E1" **\dir{-};
"E4";"E3" **\dir{-};
"E3";"E2" **\dir{-};
\endxy
\end{equation*}
In general we have 
\begin{equation*}
\chi(\mathcal{F})= (2-1) \chi(V_f)=(2-1)(4a-3)=3-4a \ ,
\end{equation*}
and when
\begin{align*}
\qquad \qquad \qquad a &\equiv 1 \pmod{3} \ , & 3-4a &\equiv 11 \pmod{12} \ , \qquad \qquad \qquad \\
\qquad \qquad \qquad a &\equiv 2 \pmod{3} \ , & 3-4a &\equiv 7 \pmod{12} \ , \qquad \qquad \qquad \\
\qquad \qquad \qquad a &\equiv 0 \pmod{3} \ , & 3-4a &\equiv 3 \pmod{12} \ . \qquad \qquad \qquad
\end{align*}
Thus, the Milnor fibre $\mathcal{F}$ is not diffeomorphic to a smoothing of a normal Gorenstein complex surface singularity.
\end{proof}

%\addcontentsline{toc}{section}{Bibliograf'ia}

% \affiliationone{% in this example, two authors share an institution
%    H. Aguilar-Cabrera\\
%    Instituto de Matem\'aticas, Unidad Cuernavaca\\
%    Universidad Nacional Aut\'onoma de M\'exico\\
%    Mexico\\
%    Institut de Math\'ematiques de Luminy\\
%    Universit\'e de la M\'editerran\'ee\\
%    France
%    \email{haydee@matcuer.unam.mx}}
\end{document}